\documentclass[article,12pt]{amsart}

\usepackage{amsfonts}
\usepackage{amsmath}
\usepackage{graphicx}
\usepackage{color}
\usepackage{epsfig}

\numberwithin{equation}{section}
\theoremstyle{plain}
\newtheorem{thm}{Theorem}[section]
\newtheorem{rem}{Remark}[section]

\newtheorem{lem}{Lemma}[section]

\newcommand{\V}{\mathbb{V}}
\newcommand{\E}{\mathbb{E}}

\newcommand{\dR}{\mathbb{R}}
\newcommand{\dE}{\mathbb{E}}

\newcommand{\cA}{\mathcal{A}}
\newcommand{\cB}{\mathcal{B}}

\newcommand{\cN}{\mathcal{N}}

\newcommand{\cO}{\mathcal{O}}
\newcommand{\cP}{\mathcal{P}}

\newcommand{\cF}{\mathcal{F}}

\newcommand{\rI}{\mathrm{I}}
\newcommand{\veps}{\varepsilon}
\newcommand{\wh}{\widehat}

\newcommand{\ind}{\mbox{1}\kern-.25em \mbox{I}}
\font\calcal=cmsy10 scaled\magstep1
\def\build#1_#2^#3{\mathrel{\mathop{\kern 0pt#1}\limits_{#2}^{#3}}}
\def\liml{\build{\longrightarrow}_{}^{{\mbox{\calcal L}}}}
\def\limp{\build{\longrightarrow}_{}^{{\mbox{\calcal P}}}}
\def\videbox{\mathbin{\vbox{\hrule\hbox{\vrule height1ex \kern.5em
\vrule height1ex}\hrule}}}

\email{Bernard.Bercu@math.u-bordeaux1.fr}
\email{Nguyen@math.unistra.fr}
\email{Jerome.Saracco@math.u-bordeaux1.fr}

\begin{document}
\title[On the asymptotic behavior of the Nadaraya-Watson estimator]
{On the asymptotic behavior of the Nadaraya-Watson estimator associated with the recursive SIR method}
\author{Bernard Bercu}
	\address{Universit\'e de Bordeaux, Institut de Math\'ematiques de Bordeaux,
	UMR CNRS 5251, 351 cours de la lib\'eration, 33405 Talence cedex, France.}
\author{Thi Mong Ngoc Nguyen}
\address{Universit\'e de Strasbourg, Institut de Recherche Math\'ematique Avanc\'ee, UMR CNRS 7501, 7 rue Ren\'e Descartes 67084 Strasbourg cedex, France.}
\author{Jerome Saracco}
\address{Universit\'e de Bordeaux, Institut de Math\'ematiques de Bordeaux,
	UMR CNRS 5251, 351 cours de la lib\'eration, 33405 Talence cedex, France.}

\begin{abstract}
We investigate the asymptotic behavior of the Nadaraya-Watson estimator for the estimation of the regression function 
in a semiparametric regression model. On the one hand, we make use of the recursive version of the 
sliced inverse regression method for the estimation of the unknown parameter of the model. 
On the other hand, we implement a recursive Nadaraya-Watson procedure 
for the estimation of the regression function which takes into account the previous estimation of the parameter of the semiparametric regression
model. We establish the almost sure convergence as well as the asymptotic normality for our Nadaraya-Watson estimator. 
We also illustrate our semiparametric estimation procedure on simulated data. 
\end{abstract}

\maketitle


\section{INTRODUCTION}


The goal of this paper is to investigate the asymptotic behavior of the Nadaraya-Watson estimator
of the regression function $f$ in the semiparametric regression model given, for all $n\geq 1$, by
\begin{equation}
\label{SEMPAR}
Y_{n}=f(\theta^{\prime} X_{n})+\veps_{n}
\end{equation}
where $(X_n)$ is a sequence of independent and identically distributed random vectors of $\dR^p$
and the driven noise $(\veps_n)$ is a real martingale difference sequence independent
of $(X_n)$. We assume in all the sequel that the unknown $p$-dimensional parameter $\theta \neq 0$.
On the one hand, we make use of the recursive version of the sliced inverse regression (SIR) method, 
originally proposed by Li \cite{LI91} and Duan and Li \cite{DUAN91}, in order to estimate $\theta$.  
On the other hand, we estimate the unknown regression function $f$ via a
recursive Nadaraya-Watson estimator which takes into account the previous estimation 
of the parameter $\theta$. Our purpose is precisely to investigate the asymptotic behavior
of the recursive Nadaraya-Watson estimator of $f$.
\ \vspace{1ex} \\
One can find a wide range of literature on nonparametric estimation of a regression function. We refer the reader to
\cite{DEVLUG01}, \cite{NAD89}, \cite{SIL86}, \cite{TSY04} for some excellent books on density and regression function estimation.
In the classical situation without any parameter $\theta$, the almost sure convergence of the Nadaraya-Watson estimator 
\cite{NAD64}, \cite{WAT64} was proved by Noda \cite{NOD76} and its asymptotic normality  was established by
Schuster \cite{SCH72}. Moreover, Choi, Hall and Rousson \cite{CHR00} propose three data-sharpening versions of 
the Nadaraya-Watson estimator in order to reduce the asymptotic variance in the central limit theorem.
In our situation, we propose to make use of a recursive Nadaraya-Watson estimator \cite{DUF97} of $f$ which
takes into account the previous estimation of the parameter $\theta$. It is given, for all $x\in \dR^p$, by
\begin{equation}
\label{RNW}
\widehat{f}_{n}(x)=\frac{\sum_{k=1}^{n} W_{k}(x)Y_{k}}{\sum_{k=1}^{n} W_{k}(x)} 
\end{equation}
with
$$
W_{n}(x)=\frac{1}{h_{n}}K\Bigl(\frac{x-\widehat{\theta}_{n-1}^{\,\prime}X_{n}}{h_{n}}\Bigr)
$$
where the kernel $K$ is a chosen probability density function and  
the bandwidth $(h_n)$ is a sequence of positive real numbers  
decreasing to zero, such that $n h_n$ tends to infinity. For the sake of simplicity, we propose to make use of
$h_n = 1/n^{\alpha}$ with $\alpha \in\, ]0,1[$. 
The main difficulty arising here is that we have to deal with the recursive SIR estimator 
$\widehat{\theta}_{n}$ of $\theta$ inside the kernel $K$. 
\ \vspace{1ex}\\ 
The paper is organized as follows. Section $2$ is devoted to the recursive SIR estimator $\widehat{\theta}_{n}$.
Our main results on the asymptotic behavior of $\widehat{f}_{n}$ are given in Section $\!3$.
Under standard regularity assumptions on the kernel $K$,
we establish the almost sure pointwise convergence of $\widehat{f}_{n}$ together with its asymptotic normality.
Section $\!4$ contains some numerical experiments on simulated data, illustrating
the good performances of our semiparametric estimation procedure. 
All the technical proofs are postponed in Appendices A and B.


\section{ON THE RECURSIVE SIR METHOD}


From the seminal work of Li \cite{LI91} and Duan and Li \cite{DUAN91} devoted to the SIR theory, we know 
that the eigenvector associated with the maximum eigenvalue of the matrix 
$\Sigma^{-1}\Gamma$  is collinear with $\theta$ where $\Sigma=\V(X_n)$ is positive definite, $\Gamma=\V(\E(X_n|T(Y_{n}))$ and 
$T$ is a slicing of the range of $Y_{n}$ into $H$ non overlapping slices $s_1,\cdots,s_H$. 
One can observe that since  the link function $f$ is unknown in the semiparametric regression model \eqref{SEMPAR}, 
the  parameter $\theta$ is not entirely identifiable. Only its direction can be identified without assuming  additional constraints. 
Li \cite{LI91} called  effective dimension reduction (EDR),  any direction collinear with $\theta$.
Moreover, the SIR theory mainly relies on the so-called linear condition  (LC) which imposes that for all $b\in \dR^p$, 
$\E[b^{\prime}X_n | \theta^{\prime}X_n]$ is linear in $\theta^{\prime}X_n$. It means that one can find
$\alpha, \beta \in \dR$ such that
$$\E[b^{\prime}X_n | \theta^{\prime}X_n]= \alpha + \beta \theta^{\prime}X_n. \leqno (\text{LC}) $$ 
This condition is required to only hold  for the true parameter $\theta$. Since $\theta$ is unknown, it is not possible in practice to verify it
a priori. Hence, we can assume that (LC) holds for all possible values of $\theta$, which is equivalent to elliptical
symmetry of the distribution of the identically distributed sequence $(X_n)$. Finally, Hall and Li \cite{HALL93} 
mentioned that (LC) is not a severe restriction because  (LC) holds to a good
approximation in many problems as the dimension $p$ of the regression vector $X_n$ increases. 
Chen and Li \cite{CHEN98} or Cook and Ni
\cite{COOK05} also provide interesting discussions on the linear condition.
\medskip

In order to obtain a recursive version of an EDR direction estimated with SIR approach, we need  an analytic expression of the
maximum eigenvector of $\Sigma^{-1}\Gamma$. It is easily tractable when the range of $Y_{n}$ is divided into 
two non overlapping slices $s_1$ and $s_2$. Hereafter we shall assume that $H=2$. 
In this special case, it is not hard to see that 
$\Gamma=p_1z_1+p_2z_2$ where $p_h=P(Y_{n}\in s_h)$ and $z_h=\E[X_n|Y_{n}\in s_h]-\E[X_n]$ with $p_h \neq 0$ for $h=1, 2$. 
Moreover, it is straightforward to show that  the eigenvector associated to the maximum eigenvalue of 
$\Sigma^{-1}\Gamma$  can be written as 
$$\widetilde{\theta}=\Sigma^{-1}(z_1-z_2).$$ 
This vector $\widetilde{\theta}$ is therefore  an EDR direction. 
For the sake of simplicity, we identify in all the sequel the EDR direction $\widetilde{\theta}$ with $\theta$.
Our purpose is now to propose an estimator of the EDR direction $\theta$. First of all, let us recall the non recursive
SIR estimator $\widetilde{\theta}_n$ of $\theta$ given by Nguyen and Saracco \cite{NGUYEN10}. The estimator
$\widetilde{\theta}_n$ can be easily obtained  from the  sample $(X_1,Y_1),\ldots, (X_n,Y_n)$ by substituting the theoritical moments by their sample 
couterparts. More precisely, $\widetilde{\theta}_n$ is given by
\begin{equation}
\label{NONRECSIR}
\widetilde{\theta}_n=\Sigma^{-1}_n(z_{1,n}-z_{2,n})
\end{equation}
 where 
\begin{equation}
\label{SIGMAN}
\Sigma_n =\frac{1}{n}\sum_{k=1}^{n}(X_k- \overline{X}_n)(X_k- \overline{X}_n)^{\prime},
\hspace{1cm}
\overline{X}_n =\frac{1}{n}\sum_{k=1}^{n} X_k
\end{equation}
and, for $h=1,2$, $z_{h,n} = m_{h,n} -\overline{X}_n$ where
\begin{equation}
\label{MHN}
m_{h,n}=\frac{1}{n_{h,n}}\sum_{k=1}^n X_k\rI_{\{Y_{k}\in {s_h} \}},
\hspace{1cm}
n_{h,n}= \sum_{k=1}^{n} \rI_{\{Y_{k}\in {s_h} \}}.
\end{equation}
Next, we focus our attention on the recursive SIR estimator $\wh{\theta}_n$ of $\theta$ proposed by Bercu, Nguyen and Saracco 
\cite{BERCU12}, \cite{NGUYEN10}. We split the sample into two parts: the subsample of the first $(n-1)$ observations  
$(X_1,Y_1),\ldots, (X_{n-1},Y_{n-1})$, and the new observation $(X_{n}, Y_{n})$. On the one hand, the inverse of the matrix
$\Sigma_n$ given by \eqref{SIGMAN} may be recursively calculated via the Riccati equation \cite{DUF97},
\begin{equation}
\label{SIGMAN}
\Sigma_n^{-1} =\frac{n}{n-1}\Sigma_{n-1}^{-1} - \frac{n}{(n-1)(n+\rho_{n})}\Sigma^{-1}_{n-1}\Phi_{n} \Phi^{\prime}_{n} \Sigma_{n-1}^{-1}  
\end{equation}  
where   $\rho_{n}  = \Phi^{\prime}_{n}  \Sigma^{-1}_{n-1}\Phi_{n} $ and $\Phi_n=X_n- \overline{X}_{n-1}$.
On the other hand, we can also obtain the recursive form of $z_{h,n}$. As a matter of fact, we have for $h=1,2$,
\begin{equation}
\label{ZHN}
z_{h,n} =  \left \{
 \begin{array}{ll}
  z_{h^*,n-1} -  {\displaystyle \frac{1}{n}\Phi_n +  \frac{1}{n_{h^*,n-1}+1}\Phi_{h^*,n}} & \ \text{if}\ h=h^*, \vspace{1ex}\\
  z_{h,n-1} -  {\displaystyle\frac{1}{n}\Phi_n}  & \mbox{otherwise,}
   \end{array}  
  \right.
  \end{equation}
where $h^*$ denotes the slice containing the observation $Y_{n}$ and $\Phi_{h^*,n} =X_n-m_{h^*,n-1}$.
We deduce from \eqref{SIGMAN} and \eqref{ZHN} that the recursive SIR estimator $\wh{\theta}_n$ is given by
\begin{equation}  
\label{RECSIR}  
\begin{array}{ccl}
\wh{\theta}_n &= & {\displaystyle \left( \frac{n}{n-1}\right)\wh{\theta}_{n -1} 
-\frac{n}{(n-1)(n+\rho_{n})}\Sigma^{-1}_{n-1}\Phi_{n} \Phi^{\prime}_{n}\wh{\theta}_{n -1} } \vspace{1ex}\\
& - &{\displaystyle \frac{(-1)^{h^*} \,n}{(n_{h^*,n-1}+1)(n-1)}\left(\Sigma^{-1}_{n-1}- 
\frac{1}{n+\rho_{n}}\Sigma^{-1}_{n-1}\Phi_{n} \Phi^{\prime}_{n} \Sigma_{n-1}^{-1} \right)\Phi_{h^*,n}}.
\end{array}
\end{equation}

The SIR estimators $\widetilde{\theta}_n$ and $\wh{\theta}_n$ share the same asymptotic properties, 
previously established in \cite{NGUYEN10}, under the following classical hypothesis.
\vspace{1ex}
\begin{displaymath}
\begin{array}{ll}
(\text{H}_1) & \textrm{The random vectors $(X_n)$ are square integrable, independent and identically} \\
& \textrm{distributed and $(X_1,Y_1),\ldots, (X_n,Y_n)$
are independently drawn from \eqref{SEMPAR}.}
 \end{array}
\end{displaymath}

\begin{lem}
\label{LEMLGNSIR}
Assume that $(\text{LC}\,)$ and $(\text{H}_1)$ hold. Then, $\wh{\theta}_n$ converges a.s. to $\theta$,
\begin{equation}
\label{LGNSIR}
|| \wh{\theta}_n- \theta ||^2 = \mathcal{O} \left(\frac{ \log(\log n)}{n}\right)\hspace{1cm}\text{a.s.} 
\end{equation} 
In addition, we also have the asymptotic normality
\begin{equation}\label{CLTSIR}  
\sqrt{n}(\wh{\theta}_{n} - \theta)
\liml \cN (0,\Delta)
\end{equation}  
where the limiting covariance matrix $\Delta$ may be explicitely calculated.
\end{lem}


\section{MAIN RESULTS}


Our purpose is to investigate the asymptotic properties of the recursive Nadaraya-Watson estimator
$\widehat{f}_{n}$ of the link function $f$ given by \eqref{RNW}.
First of all, we assume that the kernel $K$ is a positive symmetric function, bounded with compact support, 
twice differentiable with bounded derivatives, satisfying  
$$ \int_{\dR} K(x)\,dx = 1
\hspace{1cm}\text{and}\hspace{1cm}
\int_{\dR} K^2(x)\,dx= \nu^2.
$$
Moreover, it is necessary to add the following standard hypothesis.
\vspace{1ex}
\begin{displaymath}
\begin{array}{ll}
(\text{H}_2) & \textrm{The probability density function $g$ associated with $(X_n)$ is continuous, posi-}\\
& \textrm{tive on all $\dR^p$, twice differentiable with bounded derivatives}. \\
(\text{H}_3) & \textrm{The link function $f$ is Lipschitz}. 
 \end{array}
\end{displaymath}

\noindent
Our first result deals with the almost sure convergence of the estimator $\wh{f}_{n}$.

\begin{thm}
\label{THMASCVG}
Assume that $(\text{LC}\,)$ and $(\text{H}_1)$ to $(\text{H}_3)$ hold.
In addition, suppose that the sequence $(X_{n})$ has a finite moment of order $a>2$.
Then, for any $x\in\dR$, we have 
\begin{equation} 
\label{ASCVGFN}
\lim_{n\rightarrow \infty}
\wh{f}_{n}(x)=f(x)\hspace{1cm}\text{a.s.}
\end{equation}
More precisely, if the bandwidth $(h_n)$ is given by $h_n = 1/n^{\alpha}$ with 
$0<\alpha <1/3$, 
\begin{equation} 
\label{RATEASCVGFN1}
\wh{f}_{n}(x)-f(x)=\mathcal{O}\left(n^{-\alpha}\right)+
\mathcal{O}\Bigl(n^{1/a}\sqrt{\frac{\log(\log n)}{n}}\Bigr)
\hspace{1cm}\text{a.s.}
\end{equation}
while, if $1/3\leq\alpha<1$,
\begin{equation} 
\label{RATEASCVGFN2}
\wh{f}_{n}(x)-f(x)=\mathcal{O}\left(\sqrt{n^{\alpha-1}}\log n\right)+
\mathcal{O}\Bigl(n^{1/a}\sqrt{\frac{\log(\log n)}{n}}\Bigr)
\hspace{1cm}\text{a.s.}
\end{equation}
\end{thm}
\noindent

\begin{proof} The proof is given Appendix A. 
\end{proof}

\begin{rem}
In the particular case where $(X_n)$ is a sequence of  independent random vectors of $\dR^p$ sharing the same $\cN(m, \Sigma)$ distribution
where the covariance matrix $\Sigma$ is positive definite, we can replace $n^{1/a}$ by $\log n$ into \eqref{RATEASCVGFN1} and \eqref{RATEASCVGFN2}.
Consequently, for any $x\in\dR$, we obtain that if $0<\alpha <1/3$, 
\begin{equation*} 
\wh{f}_{n}(x)-f(x)=\mathcal{O}\left(n^{-\alpha}\right)
\hspace{1cm}\text{a.s.}
\end{equation*}
while, if $1/3\leq\alpha<1$,
\begin{equation*} 
\wh{f}_{n}(x)-f(x)=\mathcal{O}\left(\sqrt{n^{\alpha-1}}\log n\right)
\hspace{1cm}\text{a.s.}
\end{equation*}
\end{rem}

\noindent
The asymptotic normality of  the estimator $\wh{f}_{n}$ is as follows.

\begin{thm}
\label{THMCLT}
Assume that $(\text{LC}\,)$ and $(\text{H}_1)$ to $(\text{H}_3)$ hold.
In addition, suppose that the sequence $(X_{n})$ has a finite moment of order $a=6$
and that the sequence $(\veps_{n})$ has a finite conditional moment of order $b>2$.
Then, as soon as the bandwidth $(h_n)$ satisfies $h_n = 1/n^{\alpha}$ with $1/3<\alpha <1$, we have
for any $x\in\dR$, the pointwise asymptotic normality
\begin{equation}
\label{CLTFN}
\sqrt{nh_n}(\wh{f}_{n}(x)-f(x)) \liml \cN\Bigl(0,\frac{\sigma^{2}\nu^2}{(1+\alpha)h(\theta,x)}\Bigr)
\end{equation}
where $h(\theta,x)$ stands for the probability density function associated with $(\theta^\prime X_n)$.
\end{thm}

\begin{proof} The proof is given Appendix B. 
\end{proof}


\section{NUMERICAL SIMULATIONS}


The goal of this Section is to illustrate via some numerical experiments the theoretical results of Section $\!3$. 
We will provide the numerical behavior of our recursive estimators combining the recursive Nadaraya-Watson estimator  
of the link function $f$ together with the recursive SIR estimator 
of the parameter $\theta$. First of all, we describe in Section 4.1 the simulated model used in the numerical study and we present  the estimation procedure, 
in particular the choice of the bandwidth parameter $\alpha$ by a cross-validation criterion.  Then, we illustrate in 
Sections 4.2 and 4.3 the almost sure convergence and the asymptotic normality of our recursive Nadaraya-Watson estimator of $f$.

\subsection{Simulated model and  estimation procedures}

We consider the semiparametric regression model given, for all $n\geq 1$, by
\begin{equation*}
Y_{n}=f(\theta^{\prime} X_{n})+\veps_{n} \leqno (\text{M})
\end{equation*}
where the link function $f$ is defined, for all $x \in \dR$, by
$$
f(x)= x \exp \Bigl( \frac{3x }{4} \Bigr).
$$
The parameter $\theta$ belongs to $\dR^p$ with $p=10$ and is given by
$$
\theta= \frac{1}{\sqrt{10}} \Bigl(1, 2,-2, - 1, 0,\ldots, 0 \Bigr).
$$
Moreover, $(X_n)$ is a sequence of independent random vectors of $\dR^p$ sharing the same $\cN(0, \rI_p)$ distribution,
while $(\veps_n)$ is a sequence of independent random variables with standard $\cN(0,1)$ distribution, independent of $(X_n)$.
In Figure~\ref{nuage}, we present two scatterplots for a sample of size $n=1000$ generated from model (M).
On the left side, one can observe the data in the ``true'' reduction subspace, that is the scatterplot of 
$(\theta^{\prime}X_1, Y_1), \ldots, (\theta^{\prime}X_n, Y_n)$ based on the ``true'' EDR direction $\theta$.
On the right side, we plot the data obtained from the estimated EDR direction $\wh{\theta}_n$ calculated
via our recursive SIR procedure, that is the scatterplot of  
$(\wh{\theta}_n^{\,\prime}X_1, Y_1), \ldots, (\wh{\theta}_n^{\,\prime} X_n, Y_n)$. 
One can clearly notice that the EDR direction  has been well estimated.

\begin{figure}[htb]
\label{nuage}
\vspace{-0.8cm}
\begin{center}
\begin{tabular}{cc} 
\includegraphics[height=8cm,width=7.3cm]{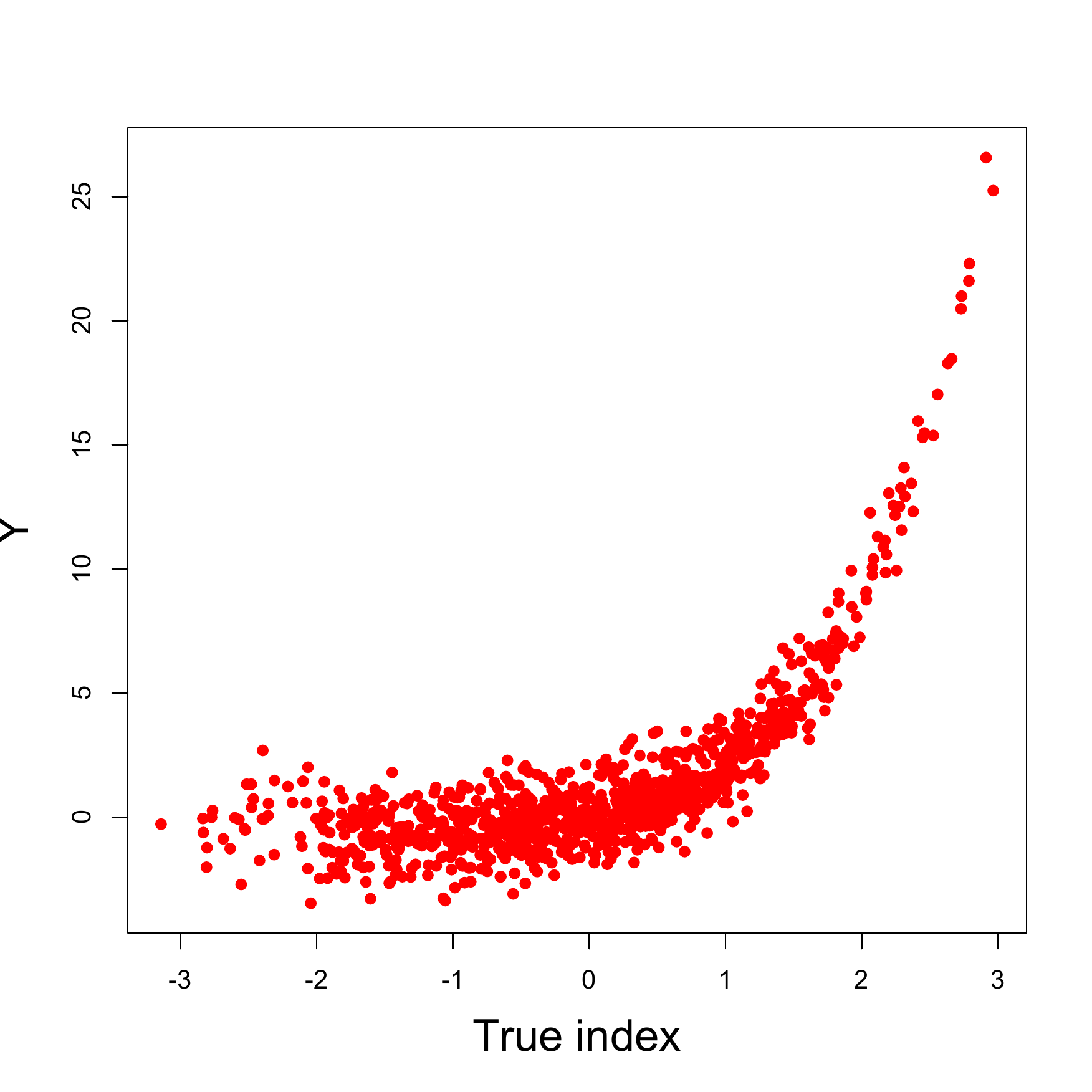}
&
\includegraphics[height=8cm,width=7.3cm]{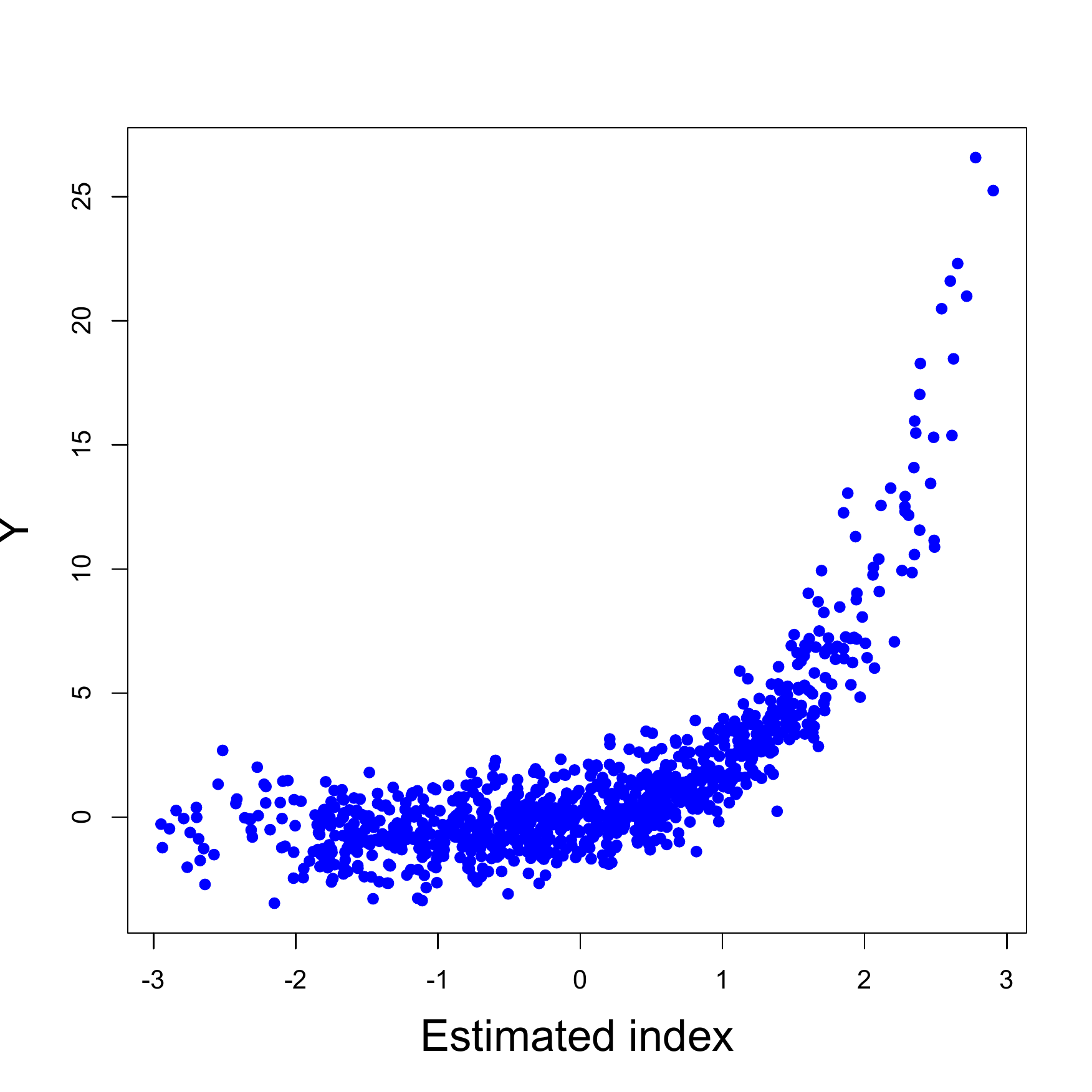}
\end{tabular}  
\end{center} 
\vspace{-0.4cm}
\small Figure 4.1. \\
Scatterplots of $(\theta^{\prime}X_1, Y_1), \ldots, (\theta^{\prime}X_n, Y_n)$ and
$(\wh{\theta}_n^{\,\prime}X_1, Y_1), \ldots, (\wh{\theta}_n^{\,\prime} X_n, Y_n)$.
\end{figure} 

\noindent
For the recursive Nadaraya-Watson estimator $\wh{f}_n$ of $f$, we have chosen the well-known Epanechnikov kernel
$$K(x)=\frac{3}{4}(1-x^2) \rI_{\{| x |\leq 1\}}$$
and the bandwidth $h_n=1/n^{\alpha}$ with $0<\alpha<1$. We now need to evaluate an optimal value 
for the smoothing parameter $\alpha$. 
The problem of deciding how much to smooth is of great importance in nonparametric regression.  We
propose to make use of the optimal data-driven bandwidth $\alpha$ which minimizes
the cross-validation criterion 
$$
CV (\alpha) = \sum_{k=p+1}^{n} (Y_k - \wh{Y}_{k,\alpha})^2
\hspace{1cm}\text{where} \hspace{1cm}
\wh{Y}_{k,\alpha}=\wh{f}_{k-1}(\wh{\theta}_{k-1}^{\,\prime}X_k).
$$
We can observe by simulations that the  $CV(\alpha)$ functions are all convex and the corresponding 
optimal data-driven bandwidth $\alpha$ lies into the interval $[0.33,0.38]$. Consequently, in all Section 4, 
we have chosen the optimal value $\alpha=0.35$.

\subsection{Almost sure convergence}
 
The good numerical performances of the recursive SIR estimator $\wh{\theta}_n$ were perviously illustrated in
\cite{BERCU12}, \cite{NGUYEN10}. In order to keep this section brief, we only focus our attention
on the almost sure convergence of $\wh{f}_n$.
We generate $N=1000$  samples of  different sizes $n=200$, $500$, $1000$, $2000$ from model (M) with $p=10$. 
For each sample, we calculate the estimation $\wh{f}_n(\wh{\theta}_n^{\prime}x)$ of $f(\theta^{\prime}x)$ for $10$ different values of $x \in \dR^p$. 
The boxplots of the $\wh{f}_n(\wh{\theta}_n^{\prime}x)$'s are given in Figure~\ref{boxplot}.
The circle point in each  boxplot 
represents the true value $f(\theta^{\prime}x)$ to easily judge the quality of the estimations. 
One can observe that the dispersion of the $\wh{f}_n(\wh{\theta}_n^{\prime}x)$'s 
are small and the mean is very close to the true value $f(\theta^{\prime}x)$. One can also notice that the larger is  the sample size $n$,  
the greater is the quality measure. As it was expected, the quality of the estimation decreases for large values of $f(\theta^{\prime}x)$ 
since the number of observations around $x$ decreases, see the scatterplots of Figure~\ref{nuage} to be convinced.
\ \vspace{1cm}\\
\begin{figure}[htb]\label{boxplot}
\vspace{-2cm}
\begin{center}
\begin{tabular}{c}  
$n=200$ \hspace{6cm} $n=500$  
\vspace{-0.5cm}\\
\includegraphics[height=7.2cm,width=7.3cm]{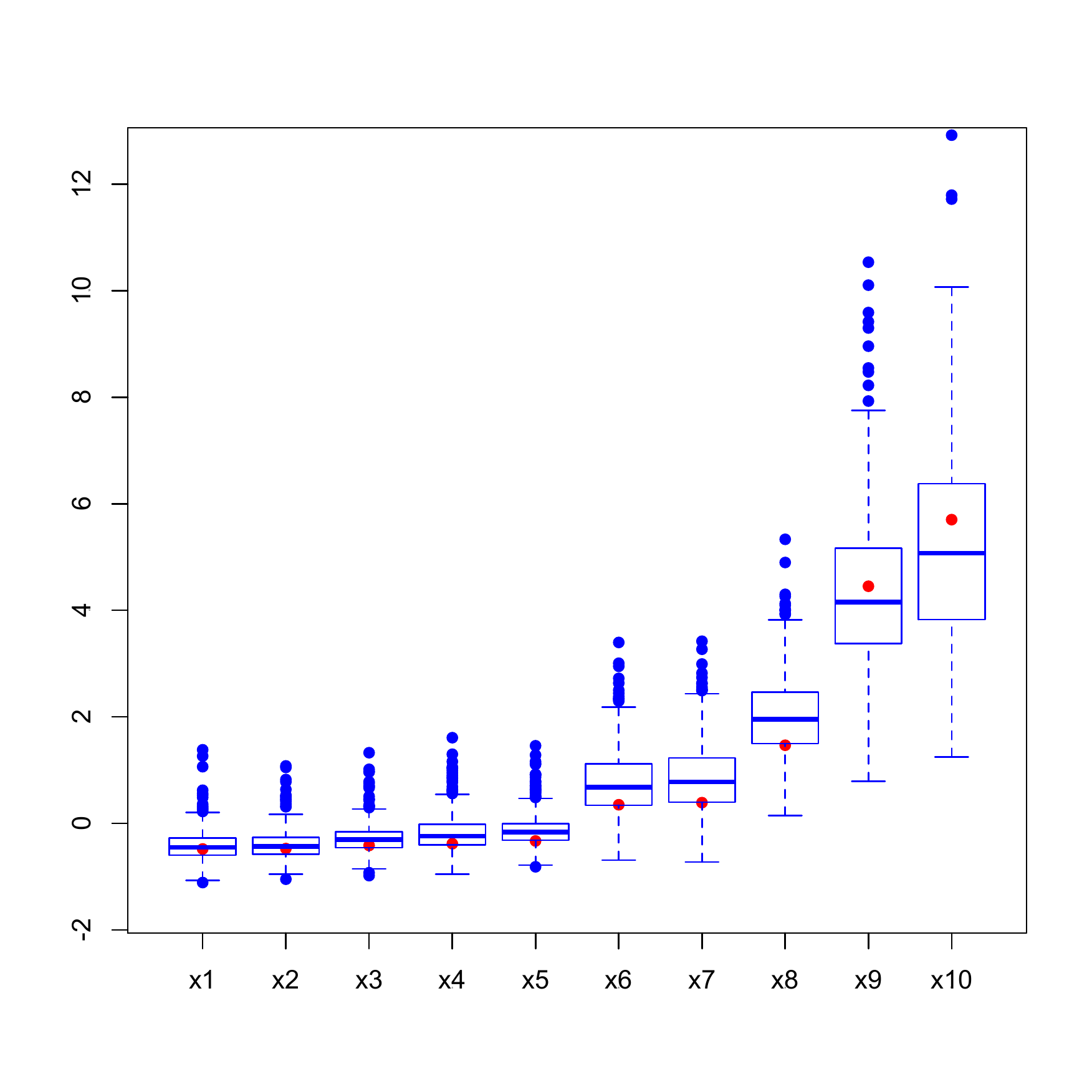}
~~~
\includegraphics[height=7.2cm,width=7.3cm]{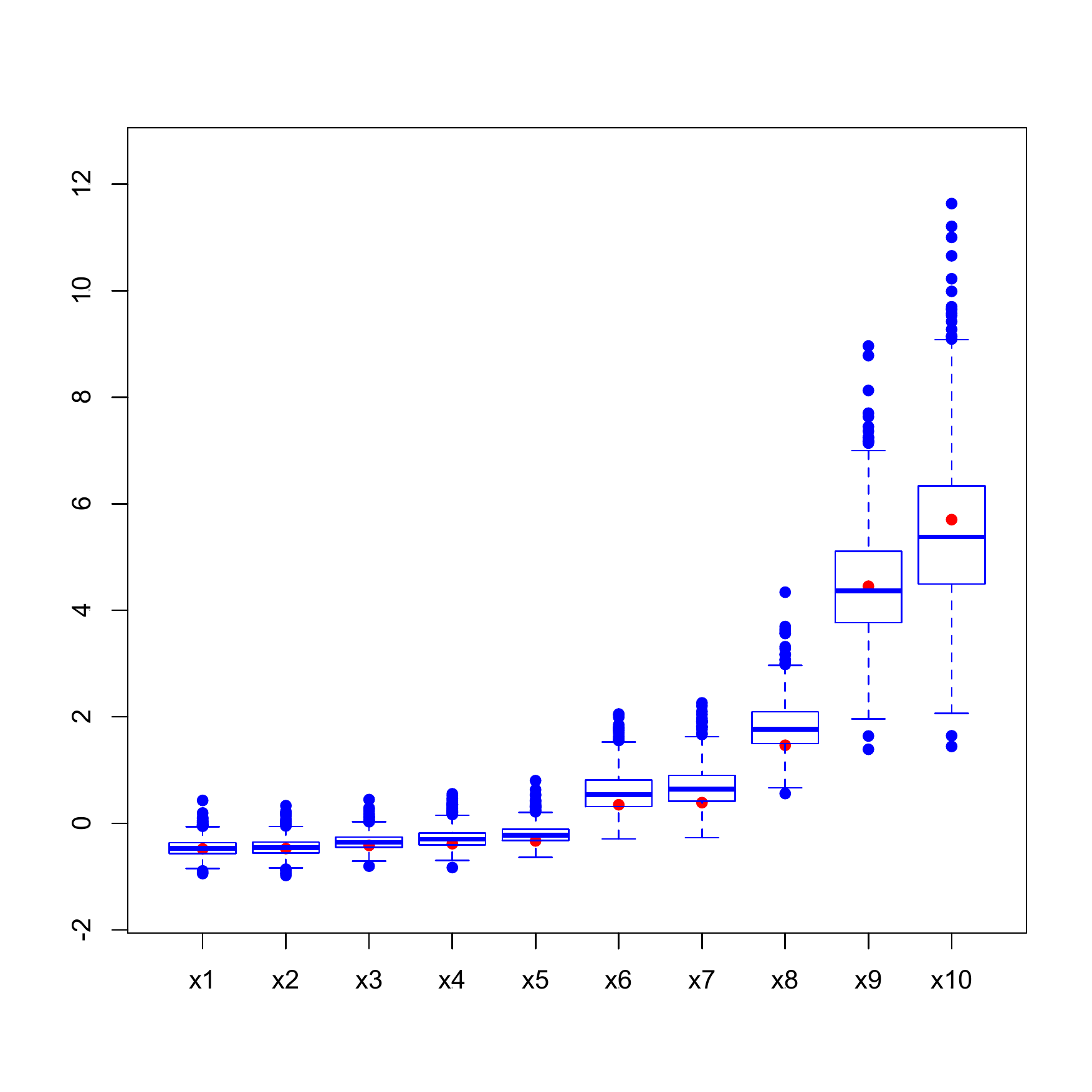}\\
$n=1000$ \hspace{6cm} $n=2000$ 
\vspace{-0.5cm} \\
\includegraphics[height=7.2cm,width=7.3cm]{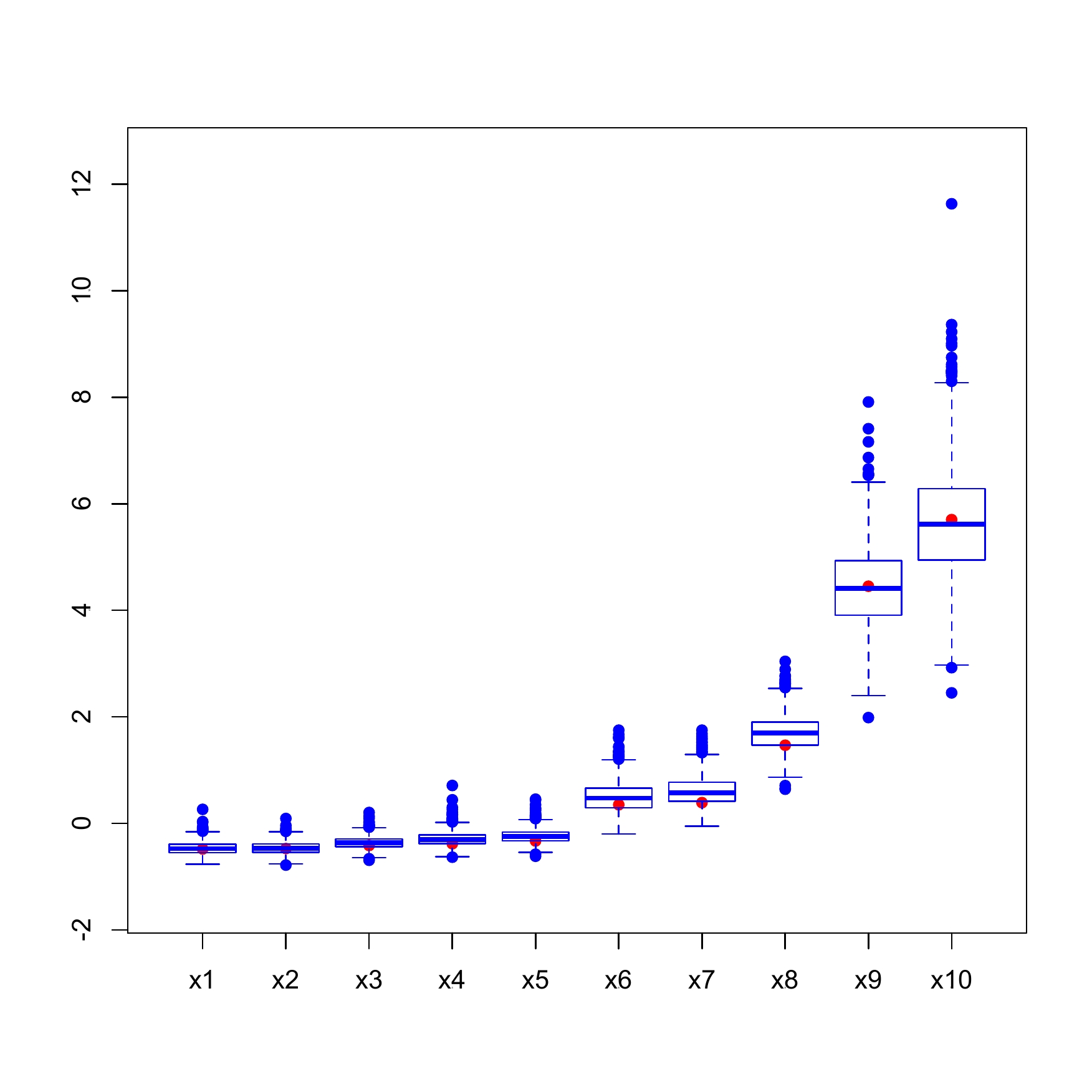}
~~~
\includegraphics[height=7.2cm,width=7.3cm]{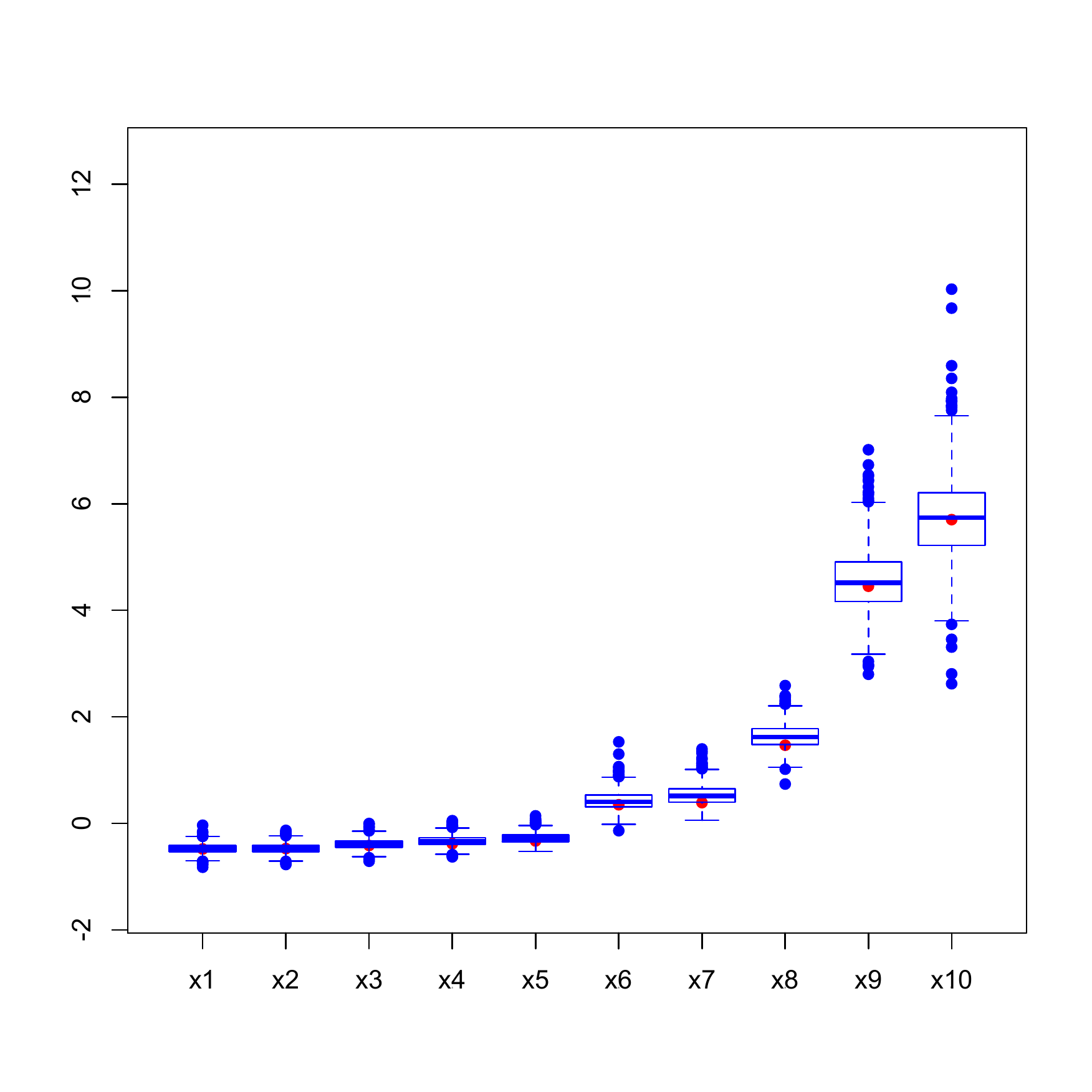}
  \end{tabular}  
\end{center} 
\vspace{-0.4cm}
\small Figure 4.2. \\
Almost sure convergence of
$\wh{f}_n(\wh{\theta}_n^{\prime}x)$ to $f(\theta^{\prime}x)$ for $10$ different values of $x$.
\end{figure}

\newpage

\subsection{Asymptotic normality}

In order to illustrate  the asymptotic normality of our recursive Nadaraya-Watson estimator, 
we generate $N=1000$ realizations of $\wh{f}_n(\wh{\theta}_n^{\prime}x)$
for $n = 1000$ from model (M) with $p=10$. In Figure~\ref{norma}, we plot the histogram of the standardized values of the
$\wh{f}_n(\wh{\theta}_n^{\prime}x)$'s for $2$ different values of $x \in \dR^p$. We add the density of the standard normal density
on each histogram. One can clearly see that the normal density coincides pretty well with all the histograms, 
which visually illustrates the asymptotic normality of our recursive Nadaraya-Watson estimator $\wh{f}_n$ of $f$.

\begin{figure}[htb]
\label{norma}
\vspace{-0.5cm}
\begin{center}
\begin{tabular}{cc} 
\includegraphics[height=8cm,width=7.3cm]{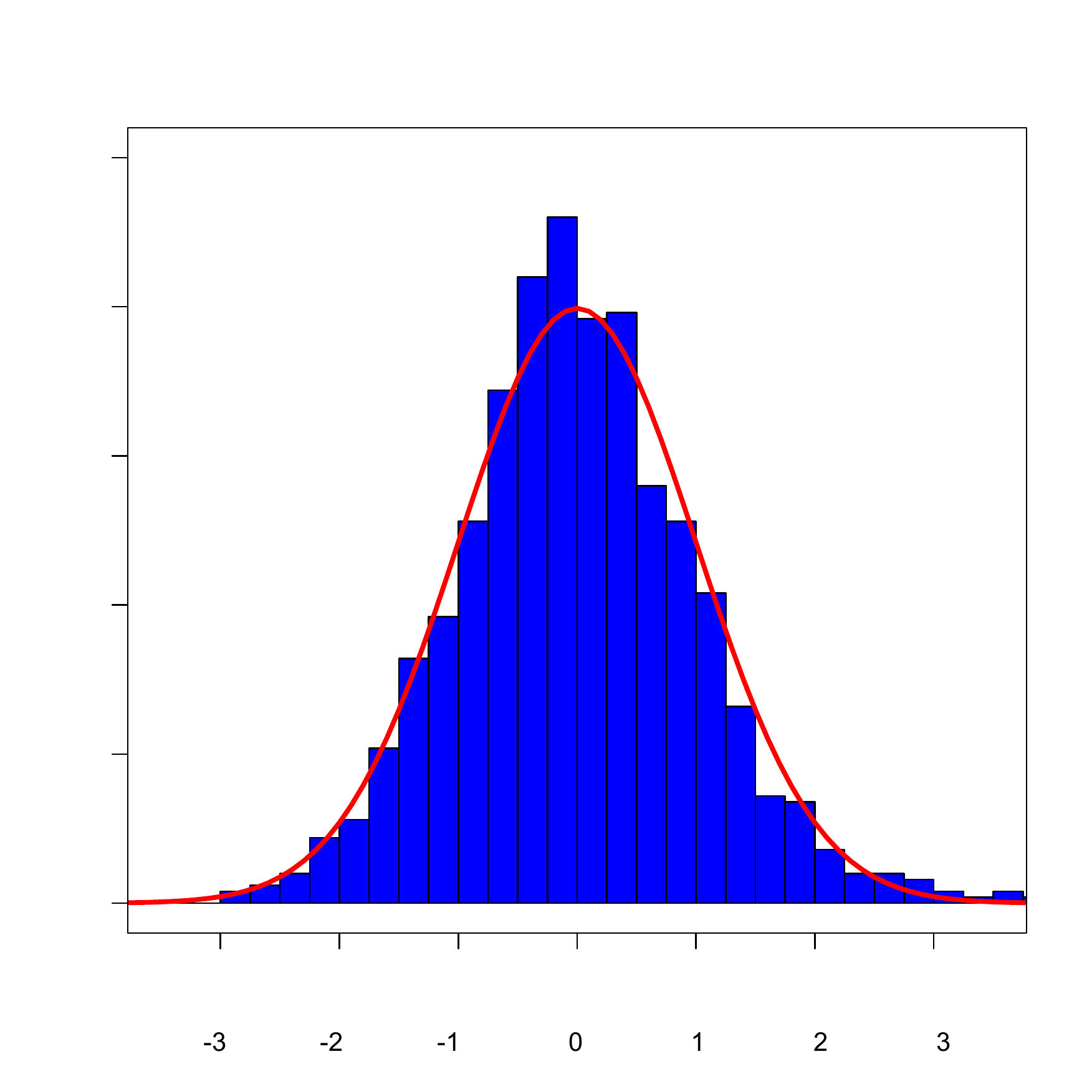}
&
\includegraphics[height=8cm,width=7.3cm]{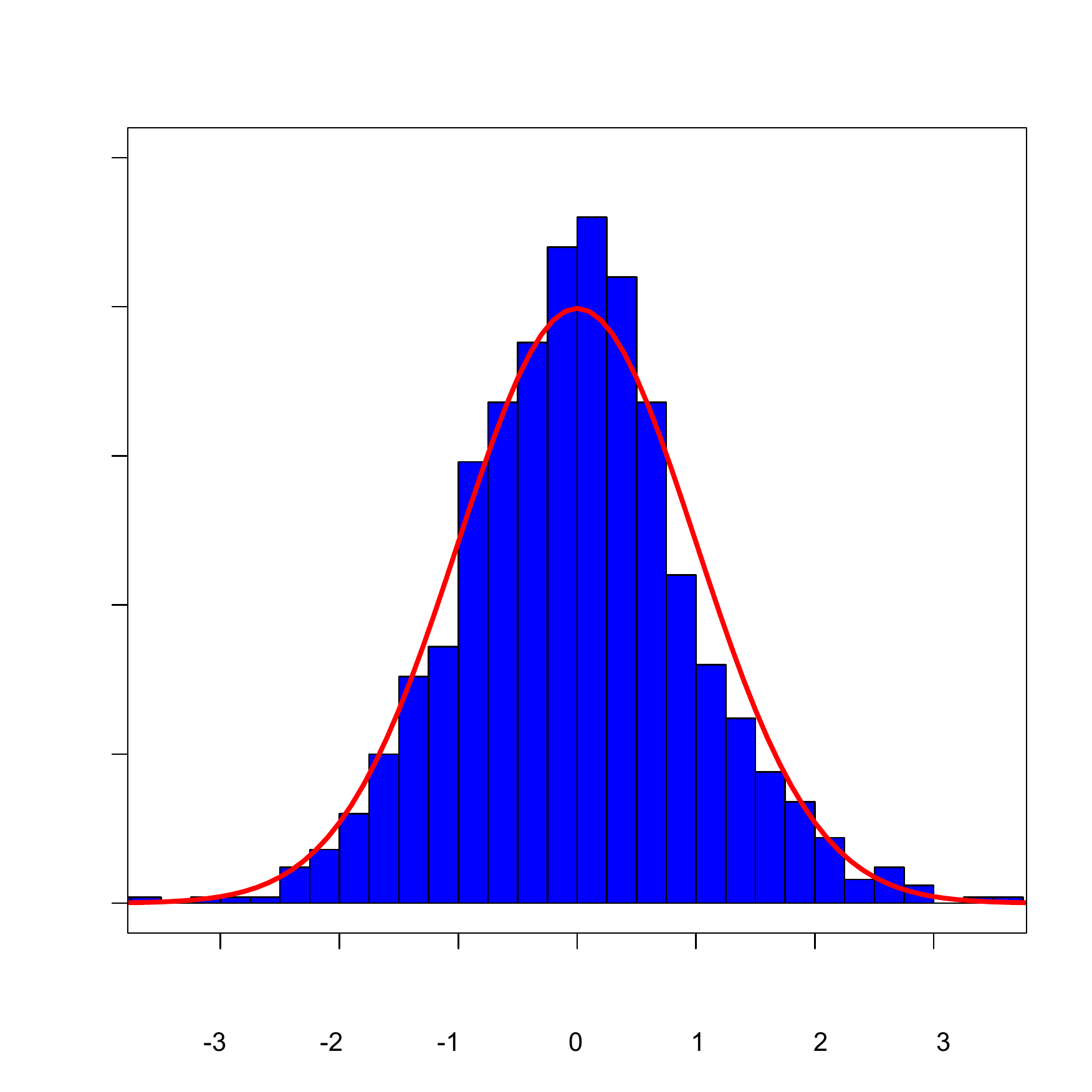}
\end{tabular}  
\end{center} 
\vspace{-0.4cm}
\small Figure 4.3. \\
Asymptotic normality of
$\wh{f}_n(\wh{\theta}_n^{\prime}x)$ to $f(\theta^{\prime}x)$ for $2$ different values of $x$.
\end{figure} 


\section*{Appendix A}

\begin{center}
{\small PROOF OF THEOREM \ref{THMASCVG}}
\end{center}

\renewcommand{\thesection}{\Alph{section}} 
\renewcommand{\theequation}
{\thesection.\arabic{equation}} \setcounter{section}{1}  
\setcounter{equation}{0}


In order to prove the almost sure pointwise convergence of Theorem  \ref{THMASCVG},
we shall denote for all $x\in \dR$
$$
P_{n}(x)=\sum_{k=1}^nW_{k}(x)\veps_{k}, \hspace{1cm} N_n(x)=\sum_{k=1}^{n} W_{k}(x),
$$
and
$$
Q_{n}(x)=\sum_{k=1}^nW_{k}(x)(f(\Phi_k)-f(x))
$$
where $\Phi_n=\theta^{\prime}X_n$. We clearly obtain from \eqref{SEMPAR} the main decomposition
\begin{equation}
\label{MAINDECO}
\wh{f}_{n}(x) -f(x)=\frac{P_n(x)+Q_n(x)}{N_n(x)}.
\end{equation}
We shall establish the asymptotic behavior of each sequence $(P_n(x))$, $(Q_n(x))$ and $(N_n(x))$.
Let $(\cF_n)$ be the filtration given by $\cF_n=\sigma(X_1,\ldots,X_n,Y_1,\ldots,Y_{n})$. First of all,
we can split $N_n(x)$ into two terms,
\begin{equation}
\label{DECON}
N_n(x)=M_n^{(N)}(x)+R_n^{(N)}(x)
\end{equation}
where
$$
M_n^{(N)}(x)=\sum_{k=1}^n\Bigl( W_{k}(x) - \dE[ W_{k}(x) | \cF_{k-1}]\Bigr)
\hspace{0.5cm} \text{and} \hspace{0.5cm}
R_n^{(N)}(x)=\sum_{k=1}^{n} \dE[ W_{k}(x) | \cF_{k-1}].
$$
On the one hand, we have
$$
\dE[W_{n}(x)|\cF_{n-1}]=\frac{1}{h_{n}}\int_{\dR^p}K\Bigl(\frac{x-\wh{\theta}_{n-1}^{\,\prime}x_n}{h_{n}}\Bigr)g(x_n)\,dx_n.
$$
We can assume without loss of generality that, for $n$ large enough, at least one component of $\wh{\theta}_{n}$ is different from zero a.s. 
As a matter of fact, we already saw from Lemma \ref{LEMLGNSIR} that $\wh{\theta}_{n}$ converges a.s. to $\theta$ which is different from zero.
For the sake of simplicity, suppose that the first component $\wh{\theta}_{n-1, 1}\neq 0$ a.s. We can make the change of variables
$$
z= \displaystyle\frac{x-\wh{\theta}^{\, \prime}_{n-1}x_n}{h_n} 
$$
and $z_2= x_{n,2}, \ldots, z_p= x_{n,p}$. The Jacobian of this linear transformation is given by
$$
J= -\frac{h_n}{\widehat \theta_{n-1,1}}.
$$
Consequently, we obtain that
\begin{equation}
\label{EXPW1}
\dE[W_{n}(x) | \cF_{n-1}]= \int_ {\dR}K(z) h(\widehat\theta_{n-1},x-zh_n)dz
\end{equation}
where
\begin{equation*}
h(\wh{\theta}_{n-1},x) =\frac{1}{|\widehat \theta_{n-1,1}|} \int_ {\dR^{p-1}}
g\Bigl(\frac{1}{\widehat \theta_{n-1,1}}\Bigl(x - \sum_{k=2}^{p} \wh{\theta}_{n-1,k}z_k\Bigr), z_2,\ldots, z_p\Bigr) dz_2\ldots dz_p.
\end{equation*}
One can observe that $h(\theta,x)$ is exactly the probability density function associated with 
the identically distributed sequence $(\theta^\prime X_n)$. Therefore, as the probability density function $g$ is
continuous, twice differentiable with bounded derivatives, we deduce from
\eqref{EXPW1} togheter with Taylor's formula that
\begin{eqnarray*}
\dE[W_{n}(x) | \cF_{n-1}] 
&=&\int_{\dR}K(z)\Bigl( h(\wh{\theta}_{n-1},x) -zh_nh^\prime(\wh{\theta}_{n-1},x)  \\
& & \qquad +
\frac{z^2h_n^2}{2}h^{\prime \prime}(\wh{\theta}_{n-1},x -zh_n\xi )\Bigr)dz, \\
&=&h(\wh{\theta}_{n-1},x)+\frac{h_{n}^{2}}{2}\int_{\dR}z^{2}K(z)
h^{\prime \prime}(\wh{\theta}_{n-1},x -zh_n\xi)dz 
\end{eqnarray*}
where $0<\xi<1$. Consequently,  for $n$ large enough,
\begin{equation} 
\vspace{1ex}
\label{sharpEXPW1}
\Bigl|\dE[W_{n}(x)|\cF_{n-1}] -h(\wh{\theta}_{n-1},x)\Bigr|\leq M_{h}\tau^{2}h_{n}^{2} \hspace{1cm}\text{a.s.}
\end{equation}
where
$$ M_{h}=\sup_{x \in \dR} \Bigl|h^{\prime \prime}(\wh{\theta}_{n-1},x )\Bigr|
\hspace{1cm} \text{and} \hspace{1cm}
\tau^{2}=\frac{1}{2} \int_{\dR}x^{2}K(x)dx.$$
Hence, we find from \eqref{sharpEXPW1} that
$$
\sum _{k=1}^{n}\Bigl|\dE[W_{k}(x)\mid \cF_{k-1}] - h(\wh{\theta}_{k-1},x)\Bigr| = \mathcal{O}\Bigl(\sum _{k=1}^{n}h^2_k\Bigr)
\hspace{1cm}\text{a.s.}
$$
It follows from the continuity of $h$ together with the fact that $\wh{\theta}_{n}$ converges to $\theta$ a.s.
and $h_n$ goes to zero that
\begin{equation} 
\label{CVGEW1}
\lim_{n \rightarrow \infty}\frac{1}{n}\sum_{k=1}^{n} \dE[W_{k}(x)|\cF_{k-1}] = h(\theta, x)
\hspace{1cm}\text{a.s.}
\end{equation}
which of course immediately implies that for all $x \in \dR$
\begin{equation} 
\label{CVGRN}
\lim_{n \rightarrow \infty}\frac{R_n^{(N)}(x)}{n}= h(\theta, x) \hspace{1cm}\text{a.s.}
\end{equation}
On the other hand, $(M_n^{(N)}(x))$ is a square integrable martingale difference sequence with predictable quadratic variation
given by
\begin{eqnarray*}
<\!M^{(N)}(x)\!>_n&=&\sum_{k=1}^n \dE[(M_k^{(N)}(x) -M_{k-1}^{(N)}(x))^2|\cF_{k-1}] ,\\
&=&\sum_{k=1}^n \Bigl( \dE[W_{k}^2(x)|\cF_{k-1}]-  \dE^2[W_{k}(x)|\cF_{k-1}]\Bigr).
\end{eqnarray*}
Via the same change of variables as in \eqref{EXPW1}, we obtain that
\begin{eqnarray*}
\dE[W_{n}^2(x)|\cF_{n-1}] &=& \frac{1}{h_n} \int_{\dR}K^2(z) h(\widehat\theta_{n-1},x-zh_n)dz,  \\
&=& \frac{1}{h_n} \int_{\dR}K^2(z)\Bigl( h(\wh{\theta}_{n-1},x) -zh_nh^\prime(\wh{\theta}_{n-1},x)  \\
& & \qquad 
+ \frac{z^2h_n^2}{2}h^{\prime \prime}(\wh{\theta}_{n-1},x -zh_n\xi )\Bigr)dz
\end{eqnarray*}
where $0<\xi<1$. Consequently,  for $n$ large enough,
\begin{equation} 
\label{sharpEXPW2}
\Bigl|\dE[W_{n}^2(x)|\cF_{n-1}] -\frac{\nu^2}{h_n} h(\wh{\theta}_{n-1},x)\Bigr|\leq M_{h}\mu^{2}h_{n} \hspace{1cm}\text{a.s.}
\end{equation}
where
$$
\nu^{2}= \int_{\dR}K^2(x)dx 
\hspace{1cm}\text{and}\hspace{1cm}
\mu^{2}= \frac{1}{2}\int_{\dR}x^2K^2(x)dx.
$$
Hence, \eqref{sharpEXPW2} ensures that
$$
\sum _{k=1}^{n}\Bigl|\dE[W_{k}^2(x)\mid \cF_{k-1}] - \frac{\nu^2}{h_k} h(\wh{\theta}_{k-1},x)\Bigr| = \mathcal{O}\Bigl(\sum _{k=1}^{n}h_k\Bigr)
\hspace{1cm}\text{a.s.}
$$
However, it is not hard to see that
$$
\lim_{n\rightarrow \infty} \frac{1}{n^{1+\alpha}}\sum_{k=1}^n \frac{1}{h_{k}} = \frac{1}{1+\alpha}.
$$
Therefore, it follows from \eqref{sharpEXPW2} together with 
the almost sure convergence of $h(\wh{\theta}_{n},x)$ to $h(\theta,x)$ and
Toeplitz's lemma that
\begin{equation}
\label{CVGEW2}
\lim_{n\rightarrow \infty}
\frac{1}{n^{1+\alpha}}\sum _{k=1}^{n} \dE[W_{k}^2(x)\mid \cF_{k-1}]
= \frac{\nu^2}{1+\alpha}h(\theta,x)  \hspace{1cm}\text{a.s.}
 \end{equation} 
Furthermore, we also have from \eqref{sharpEXPW1} that
\begin{equation} 
\label{CVGEW1SQ}
\lim_{n \rightarrow \infty}\frac{1}{n}\sum_{k=1}^{n} \dE^2[W_{k}(x)|\cF_{k-1}] = h^2(\theta, x)
\hspace{1cm}\text{a.s.}
\end{equation} 
Consequently, we deduce from \eqref{CVGEW2} and \eqref{CVGEW1SQ} that for all $x \in \dR$,
\begin{equation} 
\label{CvgBracketMN}
\lim_{n\rightarrow \infty} \frac{<\!M^{(N)}(x)\!>_n}{n^{1+\alpha}} = \frac{\nu^2}{1+\alpha}h(\theta,x)
\hspace{1cm}\text{a.s.}
\end{equation}
We are now in position to make use of the strong law of large numbers for martingales given
e.g. by Theorem 1.3.15 of \cite{DUF97}. As the probability density function $g$ is positive on
its support, we have for all $x\in \dR$, $h(\theta,x)>0$, which implies that $<\!M^{(N)}(x)\!>_n$ goes to
infinity a.s. Hence, for any $\gamma > 0$,
$(M^{(N)}_n(x))^2=o(n^{1+\alpha}(\log n)^{1+\gamma}) $ a.s.
which leads to
\begin{equation} 
\label{CVGMN}
M^{(N)}_n(x)=o(n) \hspace{1cm}\text{a.s.}
\end{equation}
Then, we obtain from \eqref{DECON}, \eqref{CVGRN} and \eqref{CVGMN} that
 for all $x \in \dR$
\begin{equation} 
\label{CVGN}
\lim_{n \rightarrow \infty}\frac{N_n(x)}{n}= h(\theta, x) \hspace{1cm}\text{a.s.}
\end{equation}
We shall now investigate the asymptotic behavior of the sequence $(P_n(x))$. Since
$(X_n)$ and $(\veps_n)$ are independent,
$(P_n(x))$ is a square integrable martingale difference sequence with predictable quadratic variation
given by
\begin{equation*}
<\!P(x)\!>_n=\sum_{k=1}^n \dE[(P_k(x) - P_{k-1}(x))^2|\cF_{k-1}] 
=\sigma^2 \sum_{k=1}^n \dE[W_{k}^2(x)|\cF_{k-1}].
\end{equation*}
Then, it follows from convergence  \eqref{CVGEW2} that
\begin{equation} 
\label{CvgBracketP}
\lim_{n\rightarrow \infty} \frac{<\!P(x)\!>_n}{n^{1+\alpha}} = \frac{\sigma^2\nu^2 }{1+\alpha}h(\theta,x)
\hspace{1cm}\text{a.s.}
\end{equation}
Consequently, we obtain from the strong law of large numbers for martingales that 
for any $\gamma > 0$ and that for all $x \in \dR$,
\begin{equation} 
\label{CVGP}
P_n(x)=o\Bigl(\sqrt{n^{1+\alpha}(\log n)^{1+\gamma}}\Bigr)=o(n) \hspace{1cm}\text{a.s.}
\end{equation}
It remains to study the asymptotic behavior of the sequence $(Q_n(x))$. We can split $Q_n(x)$ into two terms,
\begin{equation}
\label{DECOQ}
Q_n(x)=\Sigma_n(x)+\Delta_n(x)
\end{equation}
where $\wh{\Phi}_n= \wh{\theta}^{\, \prime}_{n-1}X_n$,
$$
\Sigma_n(x)=\sum_{k=1}^n W_{k}(x) ( f(\Phi_k)-f(\wh{\Phi}_k) )
\hspace{0.5cm} \text{and} \hspace{0.5cm}
\Delta_n(x)=\sum_{k=1}^{n} W_{k}(x) ( f(\wh{\Phi}_k) -f(x) ).
$$
The right-hand side of \eqref{DECOQ} is easy to handle.
As a matter of fact, the kernel $K$ is  compactly supported which means that one can find a positive constant $A$ such
that $K$ vanishes outside the interval $[-A, A]$. Thus, for all $n \geq 1$ and all $x \in \dR$,
$$
W_n(x)=\frac{1}{h_{n}}K\Bigl(\frac{x -\wh{\theta}^{\, \prime}_{n-1}X_n}{h_{n}}\Bigr)\rI_{\{|\wh{\theta}^{\, \prime}_{n-1}X_n-x| \leq A h_n \}}.
$$
In addition, the function $f$ is Lipschitz, so it exists a positive constant $C_f$ such that for all $n \geq 1$ 
\begin{equation}
\label{FLIP}
|f(\wh{\Phi}_n)-f(x)| \leq C_f |\wh{\Phi}_n-x| \leq C_f |\wh{\theta}^{\, \prime}_{n-1}X_n-x|.
\end{equation}
Consequently, we obtain from \eqref{FLIP} that for all $x \in \dR$ 
\begin{eqnarray}
|\Delta_n(x)| &\leq& C_f  \sum_{k=1}^{n}W_{k}(x)  |\wh{\theta}^{\, \prime}_{k-1}X_k-x|, \nonumber \\
&\leq & A C_f \sum_{k=1}^{n}h_k W_{k}(x).
\label{MAJDELTA}
\end{eqnarray}
Moreover, via the same lines as in the proof of \eqref{CVGEW1}, we find that
\begin{equation} 
\label{CVGHEW1}
\lim_{n \rightarrow \infty}\frac{1}{n^{1- \alpha}}\sum_{k=1}^{n} h_k\dE[W_{k}(x)|\cF_{k-1}] = \frac{1}{1- \alpha}h(\theta, x)
\hspace{1cm}\text{a.s.}
\end{equation}
Furthermore, denote
$$
M_n^{(\Delta)}(x)=\sum_{k=1}^nh_k\Bigl( W_{k}(x) - \dE[ W_{k}(x) | \cF_{k-1}]\Bigr).
$$
One can observe that $(M_n^{(\Delta)}(x))$ is a square integrable martingale difference sequence with 
bounded increments and predictable quadratic variation given by
\begin{eqnarray*}
<\!M^{(\Delta)}(x)\!>_n&=&\sum_{k=1}^n \dE[(M_k^{(\Delta)}(x) -M_{k-1}^{(\Delta)}(x))^2|\cF_{k-1}] ,\\
&=&\sum_{k=1}^n h_k^2 \Bigl( \dE[W_{k}^2(x)|\cF_{k-1}]-  \dE^2[W_{k}(x)|\cF_{k-1}]\Bigr).
\end{eqnarray*}
Hence, it follows from \eqref{sharpEXPW1} and \eqref{sharpEXPW2} together with 
the almost sure convergence of $h(\wh{\theta}_{n},x)$ to $h(\theta,x)$ and
Toeplitz's lemma that
\begin{equation} 
\label{CvgBracketMD}
\lim_{n\rightarrow \infty} \frac{<\!M^{(\Delta)}(x)\!>_n}{n^{1-\alpha}} = \frac{\nu^2 }{1-\alpha}h(\theta,x)
\hspace{1cm}\text{a.s.}
\end{equation}
Consequently, we obtain from the strong law of large numbers for martingales that 
\begin{equation} 
\label{CVGMD}
\Bigl(M_n^{(\Delta)}(x)\Bigr)^2=\cO\Bigl(n^{1-\alpha}\log n \Bigr) \hspace{1cm}\text{a.s.}
\end{equation}
Then, we infer from the conjunction of \eqref{MAJDELTA}, \eqref{CVGHEW1} and \eqref{CVGMD}
that for all $x \in \dR$
\begin{equation} 
\label{CVGDELTA}
|\Delta_n(x)|=\cO\Bigl(n^{1-\alpha} \Bigr) \hspace{1cm}\text{a.s.}
\end{equation}
The left-hand side of \eqref{DECOQ} is much more difficult to handle. We can use once again the assumption
that the function $f$ is Lipschitz to deduce that it exists a positive constant $C_f$ such that for all $n \geq 1$ 
\begin{equation}
\label{FLIPN}
|f(\wh{\Phi}_n)-f(\Phi_n)| \leq C_f |\pi_n| 
\end{equation} 
where $\pi_n=(\wh{\theta}_{n-1}-\theta)^{\prime}X_n$. Hence, it immediately follows from \eqref{FLIPN} that for all $x \in \dR$ 
\begin{equation}
|\Sigma_n(x)| \leq C_f  \sum_{k=1}^{n}W_{k}(x)  |\pi_k|.
\label{MAJSIGMA1}
\end{equation}
Denote
$$
\cA_n = \Bigl\{ | \wh{\theta}_{n-1}^{\,\prime}X_n -x | \leq A h_n \Bigr\}
\hspace{0.5cm}\text{and}\hspace{0.5cm}
\cB_n = \Bigl\{ | \theta^{\prime}X_n -x | \leq A h_n +b_n \Bigr\}
$$
where $(b_n)$ is a sequence of positive real numbers
which will be explicitely given later.  On the one hand, we immediately
have from the triangle inequality that on the set $\cA_n \cap \cB_n$, 
$$
| \pi_n| \leq 2Ah_n + b_n.
$$
On the other hand, we also have on the set $\cA_n \cap \overline{\cB_n}$,
$$
A h_n + b_n < | \theta^{\prime}X_n -x | \leq | \pi_n | + | \wh{\theta}_{n-1}^{\,\prime}X_n -x | \leq  | \pi_n | + Ah_n
$$
which implies that $| \pi_n | > b_n$. Consequently, we obtain from \eqref{MAJSIGMA1} that
\begin{equation}
|\Sigma_n(x)| \leq 2A C_f  \sum_{k=1}^{n}h_kW_{k}(x) \!+\! C_f  \sum_{k=1}^{n}b_kW_{k}(x)
\!+\! C_f \sum_{k=1}^{n}W_{k}(x)  |\pi_k| \rI_{\{ | \pi_k | > b_k \}}.
\label{MAJSIGMA2}
\end{equation}
We already saw from \eqref{CVGDELTA} that
\begin{equation} 
\label{CVGSIG1}
\sum_{k=1}^{n}h_kW_{k}(x)=\cO\Bigl(n^{1-\alpha} \Bigr) \hspace{1cm}\text{a.s.}
\end{equation}
Moreover, it is assumed that the sequence $(X_{n})$ has a finite moment of order $a>2$ which
ensures that
$$
\sup_{1 \leq k \leq n} || X_k||=o(n^{1/a}) \hspace{1cm}\text{a.s.}
$$
Consequently, we find from Lemma \eqref{LEMLGNSIR} that
\begin{equation}
\label{CVGPIN}
|\pi_n|= o(b_n)\hspace{1cm}\text{a.s.} 
\end{equation}
where we can choose
$$
b_n= n^{1/a}\sqrt{\frac{\log(\log n)}{n}}.
$$
Therefore, we clearly have
\begin{equation} 
\label{CVGSIG2}
\sum_{k=1}^{n}W_{k}(x)  |\pi_k| \rI_{\{ | \pi_k | > b_k \}}< +\infty \hspace{1cm}\text{a.s.}
\end{equation}
Furthermore, it is not hard to see that
$$
\sum_{k=1}^n b_k = \cO\Bigl(n^{1/a}\sqrt{n \log(\log n)} \Bigr).
$$
Hence, via the same lines as in the proof of \eqref{CVGDELTA}, we obtain that
\begin{equation} 
\label{CVGSIG3}
\sum_{k=1}^{n}b_kW_{k}(x)=\cO\Bigl(n^{1/a}\sqrt{n \log(\log n)} \Bigr) \hspace{1cm}\text{a.s.}
\end{equation}
Then, we deduce from the conjunction of \eqref{MAJSIGMA2}, \eqref{CVGSIG1}, \eqref{CVGSIG2}, and
\eqref{CVGSIG3} that
\begin{equation}
|\Sigma_n(x)| = \cO\Bigl(n^{1-\alpha} \Bigr) + \cO\Bigl(n^{1/a}\sqrt{n \log(\log n)} \Bigr) \hspace{1cm}\text{a.s.}
\label{MAJSIGMA3}
\end{equation}
Consequently, we infer from \eqref{CVGDELTA} and \eqref{MAJSIGMA3} that for all $x \in \dR$
\begin{equation} 
\label{CVGQ}
Q_n(x)= \cO\Bigl(n^{1-\alpha} \Bigr) + \cO\Bigl(n^{1/a}\sqrt{n \log(\log n)} \Bigr) \hspace{1cm}\text{a.s.}\hspace{1cm}\text{a.s.}
\end{equation}
Finally, we can conclude from \eqref{MAINDECO} together with \eqref{CVGN}, \eqref{CVGP} and \eqref{CVGQ} that
\begin{equation*} 
\lim_{n\rightarrow \infty}
\wh{f}_{n}(x)=f(x)\hspace{1cm}\text{a.s.}
\end{equation*}
with the almost sure rates of convergence given by \eqref{RATEASCVGFN1} and \eqref{RATEASCVGFN2}, which
completes the proof of Theorem \ref{THMASCVG}.
$\hfill 
\mathbin{\vbox{\hrule\hbox{\vrule height1.5ex \kern.6em
\vrule height1.5ex}\hrule}}$


\section*{Appendix B}

\begin{center}
{\small PROOF OF THEOREM \ref{THMCLT}}
\end{center}

\renewcommand{\thesection}{\Alph{section}} 
\renewcommand{\theequation}
{\thesection.\arabic{equation}} \setcounter{section}{2} 
\setcounter{equation}{0}


We already saw that $(P_n(x))$ is a square integrable martingale difference sequence with predictable quadratic variation
satisfying
\begin{equation*} 
\lim_{n\rightarrow \infty} \frac{<\!P(x)\!>_n}{n^{1+\alpha}} = \frac{\sigma^2\nu^2 }{1+\alpha}h(\theta,x)
\hspace{1cm}\text{a.s.}
\end{equation*}
In order to establish the asymptotic normality of Theorem \ref{THMCLT}, it is
necessary to prove that the sequence $(P_n(x))$ satisfies the Lindeberg condition, that is for all $\veps >0$,
\begin{equation}
\label{LINDEBERG1}
\cP_n(x)=\frac{1}{n^{1+\alpha}} \sum_{k=1}^{n} \dE \left[ | \Delta P_{k}(x) |^2 \rI_{\{| \Delta P_{k}(x) |  \geq \veps \sqrt{n^{1+\alpha}}\}} | \cF_{k-1} \right] \limp 0
\end{equation}
where $\Delta P_{n}(x)=P_{n}(x)-P_{n-1}(x)$. We have assumed that
the sequence $(\varepsilon_n)$ has a finite conditional moment of order $b>2$ which means that
$$
\sup_{n \geq 0} \dE[| \veps_{n} |^b|\cF_{n-1}] < + \infty \hspace{1cm}\text{a.s.}
$$
Consequently, for all $\veps >0$, we have 
\begin{eqnarray} 
\cP_n(x) & \leq &
\frac{1}{\veps^{b-2} n^c}\sum_{k=1}^{n}\dE[| \Delta P_{k}(x) |^b|\cF_{k-1}], \nonumber \\
& \leq & \frac{1}{\veps^{b-2} n^c}\sum_{k=1}^{n}\dE[ W_{k}^b(x)|\cF_{k-1}] \dE[| \veps_{k} |^b|\cF_{k-1}], \nonumber \\
& \leq & \frac{1}{\veps^{b-2} n^c} \sup_{1 \leq k \leq n} \dE[| \veps_{k} |^b|\cF_{k-1}]\sum_{k=1}^{n}\dE[ W_{k}^b(x)|\cF_{k-1}] 
\label{LINDEBERG2}
\end{eqnarray}
where $c=b(1+\alpha)/2$. In addition, via the same lines as in the proof of \eqref{CVGEW2}, we obtain that
\begin{equation}
\label{CVGEWB}
\lim_{n\rightarrow \infty}
\frac{1}{n^{1+\alpha(b-1)}}\sum _{k=1}^{n} \dE[W_{k}^b(x)\mid \cF_{k-1}]
= \frac{\xi^b}{1+\alpha(b-1)}h(\theta,x)  \hspace{1cm}\text{a.s.}
 \end{equation} 
 where
$$
\xi^b= \int_\dR K^b(x)\, dx.
$$
Therefore, we deduce from \eqref{LINDEBERG1} together with  \eqref{LINDEBERG2} and \eqref{CVGEWB} that,
for all $\veps >0$,
$$
\cP_n(x)= \cO (n^d) \hspace{1cm}\text{a.s.}
$$
where $d=(2-b)(1-\alpha)/2$. We recall that $b>2$ which means that
$d<0$. It ensures that the Lindeberg condition is satisfied. Hence, it follows
from the central limit theorem for martingales given e.g. by Corollary 2.1.10 of \cite{DUF97}
that for all $x \in \dR$,
\begin{equation}
\label{CLTP}
\frac{P_{n}(x)}{\sqrt{n^{1+\alpha}}}\liml \cN\Bigl(0,\frac{\sigma^2\nu^2 }{1+\alpha}h(\theta,x)\Bigr).
\end{equation}
Furthermore, as soon as $a \geq 6$ and $1/3<\alpha <1$, we clearly obtain from
\eqref{CVGQ} that
\begin{equation} 
\label{CVGNEWQ}
\lim_{n\rightarrow \infty} \frac{Q_n(x)}{\sqrt{n^{1+\alpha}}} = 0 \hspace{1cm}\text{a.s.}
\end{equation}
Finally, we find from \eqref{MAINDECO} together with \eqref{CVGN}, \eqref{CLTP}, \eqref{CVGNEWQ} and
Slutsky's lemma that, for all $x \in \dR$,
\begin{equation*}
\sqrt{nh_n}(\wh{f}_{n}(x)-f(x)) \liml \cN\Bigl(0,\frac{\sigma^{2}\nu^2}{(1+\alpha)h(\theta,x)}\Bigr)
\end{equation*}
which acheives the proof of Theorem \ref{THMCLT}.
$\hfill 
\mathbin{\vbox{\hrule\hbox{\vrule height1.5ex \kern.6em
\vrule height1.5ex}\hrule}}$

\ \vspace{-1ex} \\
{\bf Acknowledgements.}                              
The first author would like to thanks Bruno Portier for helpful remarks made on
a preliminary version of the paper.


\end{document}